\documentclass[12pt,reqno]{smfart}
\usepackage{amsmath,amssymb,smfthm,stmaryrd,epic,enumerate}
\usepackage[frenchb]{babel}
\usepackage[latin1]{inputenc}
\usepackage[all]{xy}
\usepackage{a4wide}

\usepackage{lmodern}

\NumberTheoremsIn{section}\SwapTheoremNumbers

\begin{document}
\newtheorem{prop-defi}[smfthm]{Proposition-DÈfinition}
\newtheorem{notas}[smfthm]{Notations}
\newtheorem{nota}[smfthm]{Notation}
\newtheorem{defis}[smfthm]{DÈfinitions}
\newtheorem{hypo}[smfthm]{HypothËse}

\def\Xm{{\mathbb X}}
\def\Tm{{\mathbb T}}
\def\Am{{\mathbb A}}
\def\Fm{{\mathbb F}}
\def\Mm{{\mathbb M}}
\def\Nm{{\mathbb N}}
\def\Pm{{\mathbb P}}
\def\Qm{{\mathbb Q}}
\def\Zm{{\mathbb Z}}
\def\Dm{{\mathbb D}}
\def\Cm{{\mathbb C}}
\def\Rm{{\mathbb R}}
\def\Gm{{\mathbb G}}
\def\Lm{{\mathbb L}}
\def\Km{{\mathbb K}}
\def\Om{{\mathbb O}}
\def\Em{{\mathbb E}}

\def\BC{{\mathcal B}}
\def\QC{{\mathcal Q}}
\def\TC{{\mathcal T}}
\def\ZC{{\mathcal Z}}
\def\AC{{\mathcal A}}
\def\CC{{\mathcal C}}
\def\DC{{\mathcal D}}
\def\EC{{\mathcal E}}
\def\FC{{\mathcal F}}
\def\GC{{\mathcal G}}
\def\HC{{\mathcal H}}
\def\IC{{\mathcal I}}
\def\JC{{\mathcal J}}
\def\KC{{\mathcal K}}
\def\LC{{\mathcal L}}
\def\MC{{\mathcal M}}
\def\NC{{\mathcal N}}
\def\OC{{\mathcal O}}
\def\PC{{\mathcal P}}
\def\UC{{\mathcal U}}
\def\VC{{\mathcal V}}
\def\XC{{\mathcal X}}
\def\SC{{\mathcal S}}

\def\BF{{\mathfrak B}}
\def\AF{{\mathfrak A}}
\def\GF{{\mathfrak G}}
\def\EF{{\mathfrak E}}
\def\CF{{\mathfrak C}}
\def\DF{{\mathfrak D}}
\def\JF{{\mathfrak J}}
\def\LF{{\mathfrak L}}
\def\MF{{\mathfrak M}}
\def\NF{{\mathfrak N}}
\def\XF{{\mathfrak X}}
\def\UF{{\mathfrak U}}
\def\KF{{\mathfrak K}}
\def\FF{{\mathfrak F}}

\def \hi{\HC}

\def \longmapright#1{\smash{\mathop{\longrightarrow}\limits^{#1}}}
\def \mapright#1{\smash{\mathop{\rightarrow}\limits^{#1}}}
\def \lexp#1#2{\kern \scriptspace \vphantom{#2}^{#1}\kern-\scriptspace#2}
\def \linf#1#2{\kern \scriptspace \vphantom{#2}_{#1}\kern-\scriptspace#2}
\def \linexp#1#2#3 {\kern \scriptspace{#3}_{#1}^{#2} \kern-\scriptspace #3}

\def \Ext{\mathop{\mathrm{Ext}}\nolimits}
\def \ad{\mathop{\mathrm{ad}}\nolimits}
\def \sh{\mathop{\mathrm{Sh}}\nolimits}
\def \irr{\mathop{\mathrm{Irr}}\nolimits}
\def \FH{\mathop{\mathrm{FH}}\nolimits}
\def \FPH{\mathop{\mathrm{FPH}}\nolimits}
\def \coh{\mathop{\mathrm{Coh}}\nolimits}
\def \res{\mathop{\mathrm{res}}\nolimits}
\def \op{\mathop{\mathrm{op}}\nolimits}
\def \rec {\mathop{\mathrm{rec}}\nolimits}
\def \art{\mathop{\mathrm{Art}}\nolimits}
\def \hyp {\mathop{\mathrm{Hyp}}\nolimits}
\def \cusp {\mathop{\mathrm{Cusp}}\nolimits}
\def \scusp {\mathop{\mathrm{Scusp}}\nolimits}
\def \Iw {\mathop{\mathrm{Iw}}\nolimits}
\def \JL {\mathop{\mathrm{JL}}\nolimits}
\def \speh {\mathop{\mathrm{Speh}}\nolimits}
\def \isom {\mathop{\mathrm{Isom}}\nolimits}
\def \Vect {\mathop{\mathrm{Vect}}\nolimits}
\def \groth {\mathop{\mathrm{Groth}}\nolimits}
\def \hom {\mathop{\mathrm{Hom}}\nolimits}
\def \deg {\mathop{\mathrm{deg}}\nolimits}
\def \val {\mathop{\mathrm{val}}\nolimits}
\def \det {\mathop{\mathrm{det}}\nolimits}
\def \rep {\mathop{\mathrm{Rep}}\nolimits}
\def \spec {\mathop{\mathrm{Spec}}\nolimits}
\def \fr {\mathop{\mathrm{Fr}}\nolimits}
\def \frob {\mathop{\mathrm{Frob}}\nolimits}
\def \ker {\mathop{\mathrm{Ker}}\nolimits}
\def \im {\mathop{\mathrm{Im}}\nolimits}
\def \Red {\mathop{\mathrm{Red}}\nolimits}
\def \red {\mathop{\mathrm{red}}\nolimits}
\def \aut {\mathop{\mathrm{Aut}}\nolimits}
\def \diag {\mathop{\mathrm{diag}}\nolimits}
\def \spf {\mathop{\mathrm{Spf}}\nolimits}
\def \Def {\mathop{\mathrm{Def}}\nolimits}
\def \twist {\mathop{\mathrm{Twist}}\nolimits}
\def \scusp {\mathop{\mathrm{Scusp}}\nolimits}
\def \Id {{\mathop{\mathrm{Id}}\nolimits}}
\def \lie {{\mathop{\mathrm{Lie~}}\nolimits}}
\def \Ind{\mathop{\mathrm{Ind}}\nolimits}
\def \ind {\mathop{\mathrm{ind}}\nolimits}
\def \loc {\mathop{\mathrm{Loc}}\nolimits}
\def \top {\mathop{\mathrm{Top}}\nolimits}
\def \ker {\mathop{\mathrm{Ker}}\nolimits}
\def \coker {\mathop{\mathrm{Coker}}\nolimits}
\def \gal {{\mathop{\mathrm{Gal}}\nolimits}}
\def \Nr {{\mathop{\mathrm{Nr}}\nolimits}}
\def \rn {{\mathop{\mathrm{rn}}\nolimits}}
\def \tr {{\mathop{\mathrm{Tr~}}\nolimits}}
\def \Sp {{\mathop{\mathrm{Sp}}\nolimits}}
\def \st {{\mathop{\mathrm{St}}\nolimits}}
\def \sp{{\mathop{\mathrm{Sp}}\nolimits}}
\def \perv{\mathop{\mathrm{Perv}}\nolimits}
\def \tor {{\mathop{\mathrm{Tor}}\nolimits}}
\def \nrd {{\mathop{\mathrm{Nrd}}\nolimits}}
\def \nilp {{\mathop{\mathrm{Nilp}}\nolimits}}
\def \obj {{\mathop{\mathrm{Obj}}\nolimits}}
\def \cl {{\mathop{\mathrm{cl}}\nolimits}}
\def \gr {{\mathop{\mathrm{gr}}\nolimits}}
\def \grr {{\mathop{\mathrm{grr}}\nolimits}}
\def \coim {{\mathop{\mathrm{Coim}}\nolimits}}
\def \can {{\mathop{\mathrm{can}}\nolimits}}
\def \Spl {{\mathop{\mathrm{Spl}}\nolimits}}

\def \rem{{\noindent\textit{Remarque:~}}}
\def \ext {{\mathop{\mathrm{Ext}}\nolimits}}
\def \End {{\mathop{\mathrm{End}}\nolimits}}

\def\semi{\mathrel{>\!\!\!\triangleleft}}
\let \DS=\displaystyle

\setcounter{secnumdepth}{3} \setcounter{tocdepth}{3}

\def \Fil{\mathop{\mathrm{Fil}}\nolimits}
\def \coFil{\mathop{\mathrm{coFil}}\nolimits}
\def \Fill{\mathop{\mathrm{Fill}}\nolimits}
\def \CoFill{\mathop{\mathrm{CoFill}}\nolimits}
\def\SF{{\mathfrak S}}
\def\PF{{\mathfrak P}}
\def \EFil{\mathop{\mathrm{EFil}}\nolimits}
\def \ECoFil{\mathop{\mathrm{ECoFil}}\nolimits}
\def \EFill{\mathop{\mathrm{EFill}}\nolimits}
\def \FP{\mathop{\mathrm{FP}}\nolimits}

\let \longto=\longrightarrow
\let \oo=\infty

\let \d=\delta
\let \k=\kappa

\newcommand{\marque}{\addtocounter{smfthm}{1}
{\smallskip \noindent \textit{\thesmfthm}~---~}}

\renewcommand\atop[2]{\ensuremath{\genfrac..{0pt}{1}{#1}{#2}}}

\title[Torsion dans la cohomologie de certaines variÈtÈs de Shimura compactes]{Sur la torsion dans la cohomologie des variÈtÈs de Shimura 
de Kottwitz-Harris-Taylor}

\alttitle{Torsion in the cohomology of Kottwitz-Harris-Taylor Shimura varieties}

\author{Pascal Boyer}

\address{UniversitÈ Paris 13, Sorbonne Paris CitÈ \\
LAGA, CNRS, UMR 7539\\ 
F-93430, Villetaneuse (France) \\
PerCoLaTor: ANR-14-CE25}


\email{boyer@math.univ-paris13.fr}






\altkeywords{torsion cohomology classes, Shimura variety, perverse sheaf, automorphic representation}


\frontmatter

\begin{abstract}
Lorsque le niveau en $l$ d'une variÈtÈ de Shimura de Kottwitz-Harris-Taylor n'est pas maximal, 
sa cohomologie ‡ coefficients dans un $\overline \Zm_l$-systËme local n'est en gÈnÈral pas libre.
Afin d'obtenir des ÈnoncÈs d'annulation de la torsion, on localise en un idÈal maximal $\mathfrak m$
de l'algËbre de Hecke. 
Nous prouvons alors un ÈnoncÈ d'annulation de la torsion de ces localisÈs, 
reposant soit sur $\mathfrak m$ directement, soit sur la reprÈsentation galoisienne  $\overline \rho_{\mathfrak m}$
qui lui est associÈe. En ce qui concerne la torsion, dans un cadre
bien moins gÈnÈral que \cite{scholze-cara}, nous obtenons de mÍme que la torsion ne fournit
pas de nouveaux systËmes de paramËtres de Satake, en prouvant que toute classe de torsion se relËve
dans la partie libre de la cohomologie d'une variÈtÈ d'Igusa.

\end{abstract}

\begin{altabstract}
When the level at $l$ of a Shimura variety of Kottwitz-Harris-Taylor is not maximal, its cohomology with
coefficients in a $\overline \Zm_l$-local system isn't in general torsion free. In order to prove torsion 
freeness results of the cohomology, we localize at a maximal ideal $\mathfrak m$ of the Hecke algebra. 
We then prove a result of torsion freeness resting either on $\mathfrak m$ itself or
on the Galois representation $\overline \rho_{\mathfrak m}$ associated to it.
Concerning the torsion, in a rather restricted case than \cite{scholze-cara}, we prove that the torsion
doesn't give new Satake parameters systems by showing that each torsion cohomology class can be raised in the free
part of the cohomology of a Igusa variety.
\end{altabstract}


\maketitle

\pagestyle{headings} \pagenumbering{arabic}

\section*{Introduction}
\renewcommand{\theequation}{\arabic{equation}}
\backmatter

Dans \cite{lan-suh}, K.-W. Lan et J. Suh prouvent un rÈsultat trËs gÈnÈral sur l'absence de torsion dans la 
cohomologie d'une variÈtÈ de Shimura PEL compacte ‡ valeur dans un $\overline \Zm_l$-systËme local. 
Pour obtenir un rÈsultat aussi gÈnÈral, les donnÈes doivent vÈrifier un certain nombre d'hypothËses 
\begin{itemize}
\item sur le systËme local qui est supposÈ trËs rÈgulier au sens de la dÈfinition 7.18 de \cite{lan-suh},

\item sur $l$ qui doit Ítre bon au sens de la dÈfinition 2.3 de \cite{lan-suh} et donc suffisamment grand
relativement au poids du systËme local considÈrÈ et

\item sur le niveau en $l$ supposÈ maximal, i.e. la variÈtÈ de Shimura est non ramifiÈe en $l$.
\end{itemize}
Ces hypothËses sont loin d'Ítre superflues comme pourra le noter le lecteur en considÈrant, par exemple,
un niveau qui est un pro-$l$-sous-groupe d'Iwahori en $l$,
et un systËme local rÈgulier $V_{\overline \Zm_l,\xi}$ tel que $V_{\overline \Fm_l,\xi}$ 
possËde des vecteurs invariants par le sous-groupe des matrices unipotentes triangulaires supÈrieures.
Alors
\begin{itemize}
\item d'une part la cohomologie de $V_{\overline \Zm_l,\xi} \otimes_{\overline \Zm_l} \overline \Qm_l$ 
est concentrÈe en degrÈ mÈdian et donc son $H^0$ est nul;

\item d'autre part 
le $H^0$ de $V_{\overline \Zm_l,\xi} \otimes_{\overline \Zm_l} \overline \Fm_l$ correspond 
aux vecteurs invariants sous le pro-$l$-sous groupe d'Iwahori en $l$ qui est donc non nul.
\end{itemize}
De ces deux faits, on en dÈduit que la torsion du $H^1$ est non nulle. En ce qui concerne le systËme local 
trivial, considÈrons la suite exacte
$$0 \rightarrow H^1 \Bigl ( G,H^0(X,\overline \Fm_l) \Bigr ) \longrightarrow H^1(Y,\overline \Fm_l)
\longrightarrow H^1(X,\overline \Fm_l)^G$$
o˘ $X \rightarrow Y$ est un revÍtement galoisien de groupe de Galois $G$, que l'on applique dans le cas o˘
\begin{itemize}
\item $X$ est une variÈtÈ de Shimura gÈomÈtriquement connexe de sorte que 
$H^0(X,\overline \Fm_l) \simeq \overline \Fm_l$,

\item $G$ est de la forme $(\Zm/l\Zm)^e$ et

\item $H^1(Y,\overline \Qm_l)$ est nul.
\end{itemize}
Ainsi comme $H^1 \bigl ( (\Zm/l\Zm)^e,\overline \Fm_l \bigr )$ est non nul,
on en dÈduit que la torsion de $H^2(Y,\overline \Zm_l)$ est non 
nulle. L'exemple le plus simple pour obtenir ces conditions consiste ‡ prendre une variÈtÈ de Shimura
de Kottwitz-Harris-Taylor pour $U(2,1)$ et deux niveaux intermÈdiaires pour que $G$ soit de la forme
voulue, cf. \cite{suh-pluri} thÈorËme 3.4.

¿ la vue de ces exemples, il semble clair que si on veut une annulation de la torsion lorsque le niveau en $l$
augmente, il est raisonnable de localiser la cohomologie en un idÈal maximal $\mathfrak m$ de l'algËbre de Hecke 
agissant sur la cohomologie. D'aprËs \cite{scholze-torsion}, est associÈe ‡ tel $\mathfrak m$ une
$\overline \Fm_l$-reprÈsentation galoisienne $\overline \rho_{\mathfrak m}$ et on cherche des conditions 
sur $\mathfrak m$ ou sur $\overline \rho_{\mathfrak m}$, pour que la localisation de la cohomologie 
en $\mathfrak m$ soit sans torsion.
Des rÈsultats dans ce sens sont obtenus dans \cite{emerton-gee}, via la thÈorie de Hodge $l$-adique
de la fibre spÈciale en $l$ de la variÈtÈ de Shimura dans le cas o˘ celle-ci est de Kottwitz-Harris-Taylor.

Dans ce travail on s'intÈresse ‡ la mÍme question mais ‡ partir de l'Ètude de la cohomologie de 
la fibre spÈciale en une place au dessus de $p \neq l$ en utilisant les calculs explicites 
de \cite{boyer-compositio}. On montre alors deux cas d'annulation de la torsion dans la cohomologie
d'une variÈtÈ de Shimura de Kottwitz-Harris-Taylor $X_I$ de niveau $I$ et de dimension relative $d-1$,
‡ coefficients dans un
systËme local $V_{\xi,\overline \Zm_l}$ associÈ ‡ une reprÈsentation irrÈductible algÈbrique $\xi$,
selon que la condition porte, thÈorËme \ref{theo1}, sur les paramËtres de Satake modulo $l$, i.e. sur 
$\mathfrak m$ directement, ou, thÈorËme \ref{theo2}, sur la reprÈsentation galoisienne 
$\overline \rho_{\mathfrak m}$ associÈe.

\medskip

\noindent \textbf{ThÈorËme A} 
\textit{Supposons qu'il existe une place $v$ non ramifiÈe pour les donnÈes, 
cf. la notation \ref{nota-spl2}, telle que le multi-ensemble des paramËtres de Satake modulo $l$ en 
$v$ associÈ ‡ $\mathfrak m$ ne contient aucun sous-multi-ensemble de la forme $\{ \alpha,q_v\alpha \}$ 
o˘ $q_v$ est le cardinal du corps rÈsiduel en $v$. Alors pour tout $i$, les localisÈs 
$H^i(X_{I,\bar \eta},V_{\xi,\overline \Zm_l})_{\mathfrak m}$ sont sans torsion.}

\medskip

\noindent \textbf{ThÈorËme B} 
\textit{Supposons qu'il existe un sous-corps $k \subset \overline \Fm_l$ tel que
$SL_n(k) \subset \overline \rho_{\mathfrak m}(G_F) \subset \overline \Fm_l^\times GL_n(k)$.
Alors si $l \geq d+2$, les localisÈs 
$H^i(X_{I,\bar \eta},V_{\xi,\overline \Zm_l})_{\mathfrak m}$ sont sans torsion pour tout $i$.}

\medskip

\rem le thÈorËme B est l'une des formes que peut prendre le thÈorËme \ref{theo2} qui utilise une 
hypothËse tirÈe
de \cite{emerton-gee} et qui est aussi vÈrifiÈe si $\overline \rho_{\mathfrak m}$ est induit d'un caractËre du groupe
de Galois d'une extension $K/F$ galoisienne cyclique.

En ce qui concerne la torsion, nous obtenons deux types de contraintes pour son existence:
\begin{itemize}
\item son apparition dans le localisÈ en $\mathfrak m$ du $i$-Ëme groupe de cohomologie, implique 
des restrictions sur le multi-ensemble des paramËtres de Satake modulo $l$ en toute place non ramifiÈe,
cf. la proposition \ref{prop1}.

\item Si on considËre le plus petit groupe de cohomologie o˘ le localisÈ en $\mathfrak m$ n'est pas
libre, alors pour une sous-reprÈsentation galoisienne irrÈductible de cette torsion, l'action du Frobenius
en toute place non ramifiÈe n'admet pas beaucoup de valeurs propres distinctes, cf. la 
proposition \ref{prop2}. En particulier si on cherche des reprÈsentations galoisiennes irrÈductibles sans
valeur propre multiple pour l'action des Frobenius, leur dimension est majorÈe explicitement, cf. le 
corollaire \ref{coro-dimension}, selon le principe que plus on s'Èloigne du degrÈ mÈdian, plus la majoration 
est contraignante.

\item 
Dans un cadre beaucoup plus simple que celui de \cite{scholze-cara}, on montre que, pour tout
$\mathfrak m$ apparaissant dans la cohomologie, le systËme
de paramËtres de Satake associÈ est la rÈduction modulo $l$ d'un systËme apparaissant dans la partie 
libre de la cohomologie d'une variÈtÈ d'Igusa sur une strate de Newton. Dans l'esprit de \cite{scholze-torsion},
on peut alors associer, d'aprËs \cite{h-t}, ‡ un tel systËme, une 
reprÈsentation galoisienne dont la rÈduction modulo $l$ est telle qu'aux
places non ramifiÈes, les frobenius annulent le polynÙme de Hecke associÈ ‡ $\mathfrak m$, cf.
le thÈorËme \ref{theo-3}. Cependant on notera bien que comme dans \cite{scholze-cara} et contrairement 
‡ \cite{scholze-torsion}, on n'obtient rien de nouveau.

\end{itemize}

Pour l'essentiel, les arguments reposent sur les calculs explicites de \cite{boyer-compositio} 
des groupes de cohomologie des strates de Newton; ces rÈsultats sont rappelÈs 
au \S \ref{para-rappel-coho}.

Les rÈsultats de ce papier et notamment le thÈorËme \ref{theo1} devraient, ‡ l'instar du corollaire 3.5.1
de \cite{emerton-gee}, permettre de prouver la partie poids de Serre de la conjecture de \cite{herzig}
pour $U(d-1,1)$, tout comme \cite{emerton-gee} le fait pour $U(2,1)$. Il faudrait pour ce faire gÈnÈraliser
les rÈsultats de \cite{egh}.

\medskip

L'auteur remercie vivement BenoÓt Stroh pour nos nombreuses discussions et en particulier pour
les exemples de classes de cohomologie de torsion ÈvoquÈes dans l'introduction. 
Enfin je remercie le rapporteur anonyme qui a permis de corriger coquilles, mal-dits et erreurs de la version
initialement soumise.

\tableofcontents

\mainmatter

\renewcommand{\theequation}{\arabic{section}.\arabic{smfthm}}

\section{GÈomÈtrie de quelques variÈtÈs de Shimura unitaires}
\label{para-shimura}

Dans la suite $l$ et $p$ dÈsigneront deux nombres premiers distincts. 
Soit $F=F^+ E$ un corps CM avec $E/\Qm$ quadratique imaginaire, dont on fixe 
un plongement rÈel $\tau:F^+ \hookrightarrow \Rm$. Pour $v$ une place de $F$, on notera 
\begin{itemize}
\item $F_v$ le complÈtÈ du localisÈ de $F$ en $v$,

\item $\OC_v$ l'anneau des entiers de $F_v$,

\item $\varpi_v$ une uniformisante et

\item $q_v$ le cardinal du corps rÈsiduel $\kappa(v)=\OC_v/(\varpi_v)$.
\end{itemize}

\begin{hypo} On supposera dans la suite que $l$ est non ramifiÈ dans $E$.
\end{hypo}

Soit $B$ une algËbre ‡ 
division centrale sur $F$ de dimension $d^2$ telle qu'en toute place $x$ de $F$,
$B_x$ est soit dÈcomposÈe soit une algËbre ‡ division et on suppose $B$ 
munie d'une involution de
seconde espËce $*$ telle que $*_{|F}$ est la conjugaison complexe $c$. Pour
$\beta \in B^{*=-1}$, on note $\sharp_\beta$ l'involution $x \mapsto x^{\sharp_\beta}=\beta x^*
\beta^{-1}$ et $G/\Qm$ le groupe de similitudes, notÈ $G_\tau$ dans \cite{h-t}, dÈfini
pour toute $\Qm$-algËbre $R$ par 
$$
G(R)  \simeq   \{ (\lambda,g) \in R^\times \times (B^{op} \otimes_\Qm R)^\times  \hbox{ tel que } 
gg^{\sharp_\beta}=\lambda \}
$$
avec $B^{op}=B \otimes_{F,c} F$. 
Si $x$ est une place de $\Qm$ dÈcomposÈe $x=yy^c$ dans $E$ alors 
$$G(\Qm_x) \simeq (B_y^{op})^\times \times \Qm_x^\times \simeq \Qm_x^\times \times
\prod_{z_i} (B_{z_i}^{op})^\times,$$
o˘, en identifiant les places de $F^+$ au dessus de $x$ avec les places de $F$ au dessus de $y$,
$x=\prod_i z_i$ dans $F^+$.

\begin{nota} \label{nota-GFv}
Pour $x$ une place de $\Qm$ dÈcomposÈe dans $E$ et $z$ une place de $F^+$ au dessus
de $x$, on notera $G(F_z)$ le facteur $(B_z^{op})^\times$ de $G(\Qm_x)$ et
$G(\Am^z)$ pour $G(\Am)$ auquel on Ùte le facteur $(B_z^{op})^\times$.
De mÍme pour $T$ un ensemble de places de $\Qm$ et $x \in T$, on notera 
$T-\{ z \}$ pour dÈsigner la rÈunion des places de $T$ distinctes de $x$ avec les
places de $F$, autres que $z$, au dessus de $x$.
\end{nota}

Dans \cite{h-t}, les auteurs justifient l'existence d'un $G$ comme ci-dessus tel qu'en outre:
\begin{itemize}
\item si $x$ est une place de $\Qm$ qui n'est pas dÈcomposÈe dans $E$ alors
$G(\Qm_x)$ est quasi-dÈployÈ;

\item les invariants de $G(\Rm)$ sont $(1,d-1)$ pour le plongement $\tau$ et $(0,d)$ pour les
autres. 
\end{itemize}

\emph{On fixe} ‡ prÈsent un nombre premier $p=uu^c$ dÈcomposÈ dans $E$ tel qu'il existe une 
place $v$ de $F$ au dessus de $u$ avec 
$$(B_v^{op})^\times \simeq GL_d(F_v).$$
On note $v=v_1, v_2,\cdots, v_r$ les places de $F$ au dessus de $u$.
Pour tout $\underline m=(m_1,\cdots,m_r) \in \Zm_{\geq 0}^r$, on pose
$$\underline m^v=(0,m_2,\cdots,m_r) 
,$$
et pour tout sous-groupe compact $U^p$ de $G(\Am^{\oo,p})$, on note
$$U_v(\underline m)=U^p \times \Zm_p^\times \times \prod_{i=1}^r 
\ker \bigl ( \OC_{B_{v_i}}^\times \longto (\OC_{B_{v_i}}/\PC_{v_i}^{m_i})^\times \bigr ),$$
ainsi que
$$U^v(\underline m)=U_v(\underline m^v).$$†

\begin{nota}
On note $\Spl$ l'ensemble des places $v$ de $F$ telles que 
$p_v:=v_{|\Qm}$ est dÈcomposÈ dans $E$ et distinct de $l$, avec
$$G(\Qm_{p_v}) \simeq \Qm_{p_v}^\times \times GL_d(F_v) \times \prod_{i=2}^r (B_{v_i}^{op})^\times,$$
o˘ $p_v=v.\prod_{i=2}^r v_i$ dans $F^+$.
\end{nota}

\begin{nota}  \label{nota-spl}
Pour $v \in \Spl$, on note $\IC_v$
l'ensemble des sous-groupes compacts ouverts \og assez petits \fg{}\footnote{tel qu'il existe 
une place $x$ pour laquelle la projection de $U_v(\underline m)$ sur $G(\Qm_x)$ ne contienne aucun 
ÈlÈment d'ordre fini autre que l'identitÈ, cf. \cite{h-t} bas de la page 90}  de 
$G(\Am^\oo)$, de la forme $U_v(\underline m)$. 
Pour $I=U_v(\underline m) \in \IC_v$, on note 
\begin{itemize}
\item $I^v =U^v(\underline m)$, 

\item $n(I):=m_1$, et

\item $\Spl(I)$ l'ensemble des places $w \in \Spl$ telles que $I$ est maximal en $w$.
\end{itemize}
\end{nota}

\begin{defi}
Pour $U_v(\underline m)$ \og assez petit \fg{}, soit $X_{U_v(\underline m)}/ \spec \OC_v$ \og 
la variÈtÈ de Shimura dite de Kottwitz-Harris-Taylor associÈe ‡ $G$\fg{} construite dans \cite{h-t}.
\end{defi}

\rem $X_{U_v(\underline m)}$ est un schÈma projectif sur $\spec \OC_v$ tel
que quand $U^p$ et $\underline m$ varient, les $X_{U_v(\underline m)}$ forment un systËme projectif dont les morphismes de transition sont finis et plats. Quand $m_1=m_1'$ alors $X_{U_v(\underline m)} 
\longrightarrow X_{U_v(\underline{m'})}$ est Ètale. Par ailleurs le systËme projectif 
$$\bigl ( X_{U_v(\underline m)} \bigr )_{U^p,\underline m}$$
est naturellement muni d'une action de $G(\Am^\oo) \times \Zm$  telle que l'action d'un ÈlÈment
$w_v$ du groupe de Weil $W_v$ de $F_v$ est donnÈe par celle de $-\deg (w_v) \in \Zm$,
o˘ $\deg=\val \circ \art^{-1}$ o˘ $\art^{-1}:W_v^{ab} \simeq F_v^\times$ est
l'isomorphisme d'Artin qui envoie les Frobenius gÈomÈtriques sur les uniformisantes.

Notons $\AC$ la variÈtÈ abÈlienne universelle sur $X_{I^v,\bar s_v}$ puis 
$\GC:= \diag(1,0,\cdots,0).\AC[v^\oo]$ le groupe de Barsotti-Tate de dimension $1$ associÈ.

\begin{notas} (cf. \cite{boyer-invent2} \S 1.3)
Pour $I \in \IC_v$, on note:
\begin{itemize}
\item $X_{I,s_v}$ la fibre spÈciale de $X_I$ en $v$ et $X_{I,\bar s_v}:=X_{I,s_v} 
\times \spec \bar \Fm_p$ la fibre spÈciale gÈomÈtrique.

\item Pour tout $1 \leq h \leq d$, $X_{I,\bar s_v}^{\geq h}$ (resp. $X_{I,\bar s_v}^{=h}$)
dÈsigne la strate fermÈe (resp. ouverte) de Newton de hauteur $h$, i.e. le sous-schÈma dont 
la partie connexe du groupe de Barsotti-Tate en chacun de ses points gÈomÈtriques
est de rang $\geq h$ (resp. Ègal ‡ $h$).
\end{itemize}
\end{notas}


\begin{notas} Pour tout $1 \leq h <d$, nous utiliserons les notations suivantes:
$$i_{h+1}:X^{\geq h+1}_{I,\bar s_v} \hookrightarrow X^{\geq h}_{I,\bar s_v}, \quad
j^{\geq h}: X^{=h}_{I,\bar s_v} \hookrightarrow X^{\geq h}_{I,\bar s_v}.$$
\end{notas}

\rem par la suite, lors de l'introduction de nouvelles notations relativement ‡ des inclusions
gÈomÈtriques, nous conserverons la convention qu'une lettre $i$ (resp. $j$) dÈsigne une 
inclusion fermÈe (resp. ouverte).

Rappelons que $\LC:= (\lie \GC)^\vee$ est un fibrÈ en droite ample sur $X_{I^v,\bar s_v}$.

\begin{theo} (cf. \cite{ito2}) \label{theo-ito}
Pour tout $1 \leq h \leq d$, il existe un invariant de Hasse gÈnÈralisÈ
$$H_h \in H^0(X^{\geq h}_{I^v,\bar s_v},\LC^{(p^h-1)})$$
qui est inversible sur $X^{=h}_{I^v,\bar s_v}$ et possËde un zÈro simple sur $X^{\geq h+1}_{I^v,\bar s_v}$.
\end{theo}

\rem en particulier $X^{=h}_{I^v,\bar s_v}$ est affine et rÈguliËre. 

Pour tout $1 \leq h< d$, les strates
$X_{I,\bar s_v}^{=h}$ sont gÈomÈtriquement induites sous l'action du parabolique $P_{h,d}(F_v)$ au 
sens o˘ il existe un sous-schÈma fermÈ $X_{I,\bar s_v,1}^{=h}$ muni d'une action par correspondances
de $G(\Am^{\oo,v}) \times GL_{d-h}(F_v) \times \Zm$ tel que:
$$X_{I^,\bar s_v}^{=h} \simeq X_{I,\bar s_v,1}^{=h} \times_{P_{h,d}(\OC_v/(\varpi_v^{n(I)}))} 
GL_d(\OC_v/(\varpi_v^{n(I)})).$$

\begin{nota}
On note $X_{I,\bar s_v,1}^{\geq h}$ l'adhÈrence de $X_{I,\bar s_v,1}^{=h}$ dans 
$X_{I,\bar s_v}^{\geq h}$ et 
$$j^{\geq h}_1: X_{I,\bar s_v,1}^{=h} \hookrightarrow X_{I,\bar s_v,1}^{\geq h}.$$
\end{nota}

\section{ReprÈsentations automorphes cohomologiques}

Avant de parler de reprÈsentations automorphes, rappelons quelques notations
sur les reprÈsentations admissibles de $GL_n$ sur un corps local $K$.
Pour $P=MN$ un parabolique standard de $GL_n$ de LÈvi $M$ et de radical unipotent $N$,
on note $\delta_P:P(K) \rightarrow \overline \Qm_l^\times$ l'application dÈfinie par
$$\delta_P(h)=|\det (\ad(h)_{|\lie N})|^{-1}.$$
Pour $(\pi_1,V_1)$ et $(\pi_2,V_2)$ des reprÈsentations de respectivement $GL_{n_1}(K)$ 
et $GL_{n_2}(K)$, et $P_{n_1,n_2}$ le parabolique standard de $GL_{n_1+n_2}$ de Levi 
$M=GL_{n_1} \times GL_{n_2}$ et de radical unipotent $N$,
$$\pi_1 \times \pi_2$$
dÈsigne l'induite parabolique normalisÈe de $P_{n_1,n_2}(K)$ ‡ $GL_{n_1+n_2}(K)$ de 
$\pi_1 \otimes \pi_2$ c'est ‡ dire
l'espace des fonctions $f:GL_{n_1+n_2}(K) \rightarrow V_1 \otimes V_2$ telles que
$$f(nmg)=\delta_{P_{n_1,n_2}}^{-1/2}(m) (\pi_1 \otimes \pi_2)(m) \Bigl ( f(g) \Bigr ),
\quad \forall n \in N, ~\forall m \in M, ~ \forall g \in GL_{n_1+n_2}(K).$$
Rappelons qu'une reprÈsentation $\pi$ de $GL_n(K)$ est dite \textit{cuspidale} si elle n'est pas 
un sous-quotient d'une induite parabolique propre.

\begin{nota}
Soient $g$ un diviseur de $d=sg$ et $\pi$ une reprÈsentation cuspidale
irrÈductible de $GL_g(K)$. L'unique quotient (resp. sous-reprÈsentation) irrÈductible de
$\pi\{ \frac{1-s}{2} \} \times \pi\{\frac{3-s}{2} \} \times \cdots \times \pi\{ \frac{s-1}{2} \}$
est notÈ $\st_s(\pi)$ (resp. $\speh_s(\pi)$).
\end{nota}

\rem du point de vue galoisien, via la correspondance de Langlands locale, la reprÈsentation
$\speh_s(\pi)$ correspond ‡ la somme directe $\sigma(\frac{1-s}{2}) \oplus \cdots \oplus
\sigma(\frac{s-1}{2})$ o˘ $\sigma$ correspond ‡ $\pi$. Plus gÈnÈralement pour
$\pi$ une reprÈsentation irrÈductible quelconque de $GL_g(K)$ associÈe ‡ $\sigma$
par la correspondance de Langlands locale, on notera $\speh_s(\pi)$ la reprÈsentation
de $GL_{sg}(K)$ associÈe, par la correspondance de Langlands locale, ‡ 
$\sigma(\frac{1-s}{2}) \oplus \cdots \oplus \sigma(\frac{s-1}{2})$.

Rappelons, cf. \cite{h-t} p.97, la paramÈtrisation des reprÈsentations algÈbriques irrÈductibles
de $G$ sur $\bar \Qm_l$. Fixons pour ce faire un plongement $\sigma_0:E \hookrightarrow
\bar{\Qm}_l$ et notons $\Phi$ l'ensemble des plongements $\sigma:F \hookrightarrow
\bar \Qm_l$ dont la restriction ‡ $E$ est $\sigma_0$. 

\medskip

\noindent \textbf{Fait}: 
Il existe une bijection explicite entre les reprÈsentations algÈbriques irrÈductibles $\xi$ de $G$ 
sur $\bar \Qm_l$ et les $(d+1)$-uplets
$\bigl ( a_0, (\overrightarrow{a_\sigma})_{\sigma \in \Phi} \bigr )$
o˘ $a_0 \in \Zm$ et pour tout $\sigma \in \Phi$, on a $\overrightarrow{a_\sigma}=
(a_{\sigma,1} \leq \cdots \leq a_{\sigma,d} )$.

\medskip

Soit $K \subset \bar \Qm_l$, une extension finie de $\Qm_l$ telle que
la reprÈsentation $\iota^{-1} \circ \xi$ de plus haut poids
$\bigl ( a_0, (\overrightarrow{a_\sigma})_{\sigma \in \Phi} \bigr )$,
soit dÈfinie sur $K$; notons $W_{\xi,K}$ l'espace de cette reprÈsentation et $W_{\xi,\OC}$
un rÈseau stable sous l'action du sous-groupe compact maximal $G(\Zm_l)$, 
o˘ $\OC$ dÈsigne l'anneau des entiers de $K$.

\rem si on suppose que $\xi$ est $l$-petit, i.e. que pour tout $\sigma \in \Phi$ et pour tout
$1 \leq i < j \leq n$, on a $0 \leq a_{\tau,j}-a_{\tau,i} < l$,
alors un tel rÈseau stable est unique ‡ homothÈtie prËs.

Notons $\lambda$ une uniformisante de $\OC$ et soit
pour $n \geq 1$, un sous-groupe distinguÈ $I_n \in \IC_v$ de $I \in \IC_v$,
compact ouvert agissant trivialement sur $W_{\xi,\OC/\lambda^n}:=W_{\xi,\OC} 
\otimes_{\OC} \OC/\lambda^n$. On note alors $V_{\xi,\OC/\lambda^n}$ le faisceau sur
$X_{I}$ dont les sections sur un ouvert Ètale $T \longrightarrow X_{I}$ sont les fonctions
$$f:\pi_0 \big ( X_{I_n} \times_{X_I} T \bigr ) \longrightarrow W_{\xi,\OC/\lambda^n}$$
telles que pour tout $k \in I$ et $C \in \pi_0 \big ( X_{I_n} \times_{X_I} T \bigr )$, on a
la relation $f(Ck)=k^{-1} f(C)$.

\begin{nota} On pose alors
$$V_{\xi,\OC}=\lim_{\atop{\leftarrow}{n}} V_{\xi,\OC/\lambda^n} \hbox{ et }
V_{\xi,K}=V_{\xi,\OC} \otimes_{\OC} K.$$
On utilisera aussi la notation $V_{\xi,\bar \Zm_l}$ et $V_{\xi,\bar \Qm_l}$ pour les versions
sur $\bar \Zm_l$ et $\bar \Qm_l$ respectivement.
\end{nota}

\rem rappelons que la reprÈsentation $\xi$ est dite \emph{rÈguliËre} si son paramËtre
$\bigl ( a_0, (\overrightarrow{a_\sigma})_{\sigma \in \Phi} \bigr )$ est tel que pour 
tout $\sigma \in \Phi$, on a $a_{\sigma,1} < \cdots < a_{\sigma,d}$.

On fixe ‡ prÈsent une  $\Cm$-reprÈsentation irrÈductible algÈbrique de dimension finie $\xi$ de $G$. 

\begin{defi} \label{defi-automorphe}
Une $\Cm$-reprÈsentation irrÈductible $\Pi_{\oo}$ de $G(\Am_{\oo})$ est dite $\xi$-cohomologique s'il
existe un entier $i$ tel que
$$H^i((\lie G(\Rm)) \otimes_\Rm \Cm,U_\tau,\Pi_\oo \otimes \xi^\vee) \neq (0)$$
o˘ $U_\tau$ est un sous-groupe compact modulo le centre de $G(\Rm)$, maximal, cf. \cite{h-t} p.92. 
On notera $d_\xi^i(\Pi_\oo)$ la dimension de ce groupe de cohomologie.
Une $\bar \Qm_{l}$-reprÈsentation irrÈductible $\Pi^{\oo}$ de $G(\Am^{\oo})$ sera dit automorphe 
$\xi$-cohomologique s'il existe une $\Cm$-reprÈsentation $\xi$-cohomologique $\Pi_\oo$ de $G(\Am_\oo)$ 
telle que
$\iota_{l}\Bigl ( \Pi^{\oo} \Bigr ) \otimes \Pi_{\oo}$ est une $\Cm$-reprÈsentation automorphe de $G(\Am)$.
\end{defi}

\begin{nota} Pour $\Pi$ une reprÈsentation irrÈductible admissible de $G(\Am)$, on note
$m(\Pi)$ sa multiplicitÈ dans l'espace des formes automorphes.
\end{nota}

Pour $\Pi$ une reprÈsentation automorphe irrÈductible admissible cohomologique de $G(\Am)$,
rappelons, cf. par exemple le lemme 3.2 de \cite{boyer-aif}, que pour 
$x$ une place de $\Qm$ dÈcomposÈe $x=yy^c$ dans $E$ et $z$
une place de $F$ au dessus de $y$ telle que, avec la notation \ref{nota-GFv},
$G(F_z) := (B_z^{op})^\times \simeq GL_d(F_z)$,
la composante locale $\Pi_z$, au sens de \ref{nota-GFv},
est de la forme $\speh_s(\pi_z)$ pour $\pi_z$ une reprÈsentation 
irrÈductible non dÈgÈnÈrÈe et $s$ un entier $\geq 1$ qui ne dÈpend
que de $\Pi$ et non de la place $z$ comme ci-dessus. 

\begin{defi} \label{defi-parametre}
L'entier $s$ ci-avant est appelÈ la profondeur de dÈgÈnÈrescence de $\Pi$.
\end{defi}

\rem on rappelle qu'un telle reprÈsentation $\Pi$ est dite tempÈrÈe si sa profondeur de
dÈgÈnÈrescence $s$ est Ègale ‡ $1$.

\begin{nota}
Les $d^i_\xi(\Pi_\oo)$ sont nuls si $|d-1-i| \geq s$ ou si $d-1-i \equiv s \mod 2$.
Sinon ils sont tous Ègaux, on note $d_\xi(\Pi_\oo)$ la valeur commune non nulle des $d^i_\xi(\Pi_\oo)$.
\end{nota}

\section{Rappels sur la $\overline \Qm_l$-cohomologie d'aprËs \cite{boyer-compositio}}
\label{para-rappel-coho}

Dans ce paragraphe $v$ dÈsigne une place de $F$ telle que $p_v:=v_{|\Qm}$ est dÈcomposÈ dans $E$ 
distinct de $l$ et  $G(\Qm_{p_v}) \simeq \Qm_{p_v}^\times \times GL_d(F_v) \times 
\prod_{i=2}^r (B_{v_i}^{op})^\times$, o˘ $p_v=v.\prod_{i=2}^r v_i$ dans $F^+$.

\begin{nota} \label{nota-hixi}

Pour $1 \leq h \leq d$, on note $\IC_v(h)$ l'ensemble des sous-groupes compacts ouverts de la
forme 
$$U_v(\underline m,h):= 
U_v(\underline m^v) \times \left ( \begin{array}{cc} I_h & 0 \\ 0 & K_v(m_1) \end{array} \right ),$$
o˘ $K_v(m_1)=\ker \bigl ( GL_{d-h}(\OC_v) \longrightarrow GL_{d-h}(\OC_v/ (\varpi_v^{m_1})) \bigr )$.
La notation $[H^i(h,\xi)]$ (resp. $[H^i_!(h,\xi)]$) dÈsignera l'image de 
$$\lim_{\atop{\longrightarrow}{I \in \IC_v(h)}} H^i(X_{I,\bar s_v,1}^{\geq h}, V_{\xi,\bar \Qm_l}[d-h]) 
\qquad \hbox{resp. }
\lim_{\atop{\longrightarrow}{I \in \IC_v(h)}} H^i(X_{I,\bar s_v,1}^{\geq h}, j^{\geq h}_{1,!} V_{\xi,\bar \Qm_l}[d-h]) 
$$ 
dans le groupe de Grothendieck $\groth(v,h)$ des reprÈsentations admissibles de 
$G(\Am^{\oo}) \times GL_{d-h}(F_v) \times \Zm$.
\end{nota}

\noindent \textit{Remarques}:
\begin{itemize}
\item[(i)] comme tous les compacts de $\IC_v$ (resp. $\IC_v(h)$) contiennent le
facteur $\Zm_{p_v}^\times$, les reprÈsentations $\Pi$ qui vont intervenir par la suite, dans les diffÈrents 
groupes de cohomologie, devront toutes vÈrifier que leur composante $\Pi_{p_v,0}$ sur 
le facteur de similitude $\Qm_{p_v}^\times$, est telle $(\Pi_{p_v,0})_{|\Zm_{p_v}^\times}=1$.

\item[(ii)] L'action de $\sigma \in W_v$ sur ces $GL_{h}(F_v) \times \Zm$-modules 
est donnÈe par celle de $-\deg \sigma \in \Zm$ composÈe avec celle de $\Pi_{p_v,0}(\art^{-1} (\sigma))$.

\item[(iii)] Par ailleurs, on munit ces espaces d'une action de $GL_{h}(F_v)$ via le morphisme 
$$\val \circ \det: GL_{h}(F_v) \longrightarrow \Zm$$
et enfin une action de $P_{h,d}(F_v)$ via son facteur de LÈvi $GL_{h}(F_v) \times GL_{d-h}(F_v)$,
i.e. en faisant agir trivialement son radical unipotent.
\end{itemize}

Pour $I_0 \in \IC_v(h)$ qui est maximal en $v$, i.e. $m_1=0$, on a
\addtocounter{smfthm}{1}
\begin{equation} \label{eq-Ih1}
H^i(X_{I_0,\bar s_v}^{\geq h},V_{\xi,\overline \Qm_l})= \Bigl (
\lim_{\atop{\rightarrow}{I \in \IC_v(h)}} 
H^i(X_{I,\bar s_v,1}^{\geq h}, V_{\xi,\bar \Qm_l}) \Bigr )^{I_0}
\end{equation}
ainsi que
\addtocounter{smfthm}{1}
\begin{equation} \label{eq-Ih2}
H^i(X_{I_0,\bar s_v}^{\geq h},j^{\geq h}_{!} V_{\xi,\overline \Qm_l})= \Bigl ( 
\lim_{\atop{\rightarrow}{I \in \IC_v(h)}} 
H^i(X_{I,\bar s_v,1}^{=h},j^{\geq h}_{1,!} V_{\xi,\bar \Qm_l}) \Bigr )^{I_0}.
\end{equation}

\begin{nota}
Pour $\Pi^{\oo,v}$ une reprÈsentation irrÈductible de $G(\Am^{\oo,v})$, on notera $\groth(h)\{ \Pi^{\oo,v} \}$ 
le sous-groupe facteur direct de $\groth(v,h)$ engendrÈ par les irrÈductibles de la forme 
$\Pi^{\oo,v} \otimes \pi_{v,et} \otimes \zeta$ o˘ $\pi_{v,et}$ (resp. $\zeta$) est une reprÈsentation 
irrÈductible quelconque de $GL_{d-h}(F_v)$ (resp. de $\Zm$). On notera alors
$$[H^i(h,\xi)]\{ \Pi^{\oo,v} \}$$
la projection de $[H^i(h,\xi)]$ sur ce facteur direct.
\end{nota}

On Ècrit 
$$[H^i(h,\xi)]\{ \Pi^{\oo,v} \}=\Pi^{\oo,v} \otimes \Bigl ( \sum_{\Psi_v,\xi} m_{\Psi_v,\zeta}(\Pi^{\oo,v}) \Psi_v 
\otimes \zeta \Bigr ),$$ 
o˘ $\Psi_v$ (resp. $\xi$) dÈcrit l'ensemble des reprÈsentations admissibles de $GL_h(F_v)$,
(resp. de $\Zm$ que l'on voit comme une reprÈsentation non ramifiÈe de $W_v$).

\begin{prop} \label{prop-preHecke}
Pour $\PC$ l'ensemble des nombres premiers de $\Zm$ et
$I=\otimes_{q \in \PC} I_q \in \IC_v$ maximal en $v$, 
dans le groupe de Grothendieck des reprÈsentations de l'algËbre
de Hecke $\otimes_{q \in \PC} \overline \Qm_l[I_q \backslash G(\Qm_q)/I_q]$,
on a avec les notations prÈcÈdentes
$$[H^{d-h+i}(X_{I,\bar s_v}^{\geq h},V_{\xi,\overline \Qm_l})]=\sum_{\Pi^{\oo,v}}
(\Pi^{\oo,v} )^{I^{\oo,v}} \otimes \Bigl (  \sum_{\Psi_v,\xi} 
m_{\Psi_v,\zeta}(\Pi^{\oo,v})  \bigl ( \speh_h \zeta  \times  \Psi_v \bigr )^{GL_d(\OC_v)} \Bigr ).$$
\end{prop}

\begin{proof}
Le rÈsultat dÈcoule directement de (\ref{eq-Ih1}) avec la description de l'action de $GL_d(F_v)$,
et du fait que, d'aprËs \cite{boyer-compositio}, la dÈcomposition de
${\displaystyle \lim_{\atop{\rightarrow}{I \in \IC_v}}} H^i(X_{I,\bar s_v,1}^{\geq h}, V_{\xi,\bar \Qm_l})$
selon ses $\Pi^{\oo,v}$-composantes est semi-simple.
\end{proof}

\rem on a une ÈgalitÈ du mÍme style pour la cohomologie ‡ support compact 
$H^i(X_{I,\bar s_v}^{\geq h},j^{\geq h}_! V_{\xi,\overline \Qm_l})$.

Les rÈsultats suivants se dÈduisent directement, dans le cas de la reprÈsentation triviale, 
de la description des groupes de cohomologie
des extensions intermÈdiaires (resp. par zÈro) des systËmes locaux d'Harris-Taylor donnÈs
aux \S 3 (resp. \S 5) de \cite{boyer-compositio}. Le lecteur pourra trouver utile, la rÈÈcriture
de ces rÈsultats dans \cite{boyer-aif}.

\begin{prop} \label{prop-temperee}
Soit $\Pi$ une reprÈsentation automorphe irrÈductible $\xi$-cohomologique tempÈrÈe.
Pour tout $h=1,\cdots,d$ et pour tout $i \neq 0$,
$$[H^i(h,\xi)]\{ \Pi^{\oo,v} \} \qquad \hbox{et} \qquad [H^i_!(h,\xi)]\{ \Pi^{\oo,v} \}$$
sont nuls.
\end{prop}

\begin{proof}
Le rÈsultat pour $H^i(h,\xi)$ est un cas particulier de la proposition\footnote{laquelle dÈcoule directement
de la proposition 3.6.1 de \cite{boyer-compositio}} 3.6 de \cite{boyer-aif} pour le systËme local constant, 
i.e. $\pi_v$ est la reprÈsentation triviale, et $s=1$. Pour la cohomologie ‡ support compact, on peut soit
Èvoquer la proposition 3.12 de \cite{boyer-aif}, soit utiliser la description, d'aprËs le corollaire 5.4.1
de \cite{boyer-invent2}, de cette extension par
zÈro en termes des systËmes locaux sur les strates de Newton d'indices $h' \geq h$.

\end{proof}

\begin{prop} \label{prop-non-temperee}
Soit $\Pi$ une reprÈsentation automorphe irrÈductible $\xi$-cohomologique non tempÈrÈe de profondeur
de dÈgÈnÈrescence $s>1$ au sens de la dÈfinition \ref{defi-parametre}. Alors 
\begin{itemize}
\item[(i)] pour tout $h>s$, les $[H^i_!(h,\xi)]\{ \Pi^{\oo,v} \}$ sont nuls pour tout $i$;

\item[(ii)] pour tout $h \neq s$, $[H^0_!(h,\xi)]\{ \Pi^{\oo,v} \}$ est nul.
\end{itemize}
\end{prop}

\begin{proof}
Le rÈsultat est donnÈ ‡ la proposition \footnote{laquelle dÈcoule directement des calculs 
de \cite{boyer-compositio} \S 5} 3.12 de \cite{boyer-aif} en prenant $t=1$. En particulier pour (ii),
le rÈsultat dÈcoule du fait que, avec les notations de loc. cit., $n_{s,1}(h,0)$
est non nul si et seulement si $h=s$.
\end{proof}

\rem dans loc. cit. on montre plus prÈcisÈment que $[H^i_!(h,\xi)]\{ \Pi^{\oo,v} \} \neq (0)$
si et seulement si $i=s-h \geq 0$.

\begin{nota} 
Soit $\UC_G(\Pi^{\oo,v})$ l'ensemble des reprÈsentations irrÈductibles automorphes $\Pi'$ de $G(\Am)$ 
telles que $(\Pi')^{\oo,v} \simeq \Pi^{\oo,v}$. 
\end{nota}

\rem d'aprËs le corollaire VI.2.2 de \cite{h-t}, la composante locale 
$\Pi'_v$ d'un $\Pi' \in \UC_G(\Pi^{\oo,v})$ ne dÈpend pas de $\Pi'$ tel que $d_\xi(\Pi'_\oo) \neq 0$, 
cf. le corollaire VI.2.2 de \cite{h-t}.

On suppose pour la fin de ce paragraphe que
$\Pi$ est une reprÈsentation automorphe irrÈductible $\xi$-cohomologique tempÈrÈe dont la composante 
locale en $v$ est 
$$\Pi_v \simeq \st_{t_1}(\pi_{v,1}) \times \cdots \times \st_{t_u}(\pi_{v,u}),$$
o˘ pour $i=1,\cdots,u$, $\pi_{v,i}$ est une reprÈsentation irrÈductible cuspidale de
$GL_{g_i}(F_v)$.

\begin{prop} \label{prop-temperee-explicite}
Avec les notations et les hypothËses prÈcÈdentes concernant $\Pi$, 
\begin{itemize}
\item $[H^0(h,\xi)] \{ \Pi^{\oo,v} \}$ est nulle sauf si tous les $\pi_{v,i}$ pour $i=1,\cdots, u$ 
sont des caractËres.

\item Dans le cas o˘ pour tout $i=1,\cdots,u$, $\pi_{v,i}$ est un caractËre de 
$F_v^\times$ que l'on note $\chi_{v,i}$, on les ordonne de faÁon
que les $r$ premiers correspondent aux non ramifiÈs. On a
alors dans $\groth(h)\{ \Pi^{\oo,v} \}$, l'ÈgalitÈ
$$[H^0(h,\xi)] \{ \Pi^{\oo,v} \} = \Bigl ( \frac{\sharp \ker^1(\Qm,G)}{d} \sum_{\Pi' \in \UC_G(\Pi^{\oo,v})} 
m(\Pi') d_\xi(\Pi'_\oo) \Bigr )  \Bigl ( \sum_{1 \leq k \leq r:~ t_k=h} \Pi_v^{(k)} \otimes \chi_{v,k}
\Xi^{\frac{d-h}{2}} \Bigr )$$
o˘ 
\begin{itemize}
\item $ \ker^1(\Qm,G)$ est le sous-ensemble de $H^1(\Qm,G)$ constituÈ des ÈlÈments qui deviennent 
triviaux dans $H^1(\Qm_{p'},G)$ pour toute place $p'$ de $\Qm$,

\item $\Pi_v^{(k)}:= \st_{t_1}(\chi_{v,1}) \times \cdots \times \st_{t_{k-1}}(\chi_{v,k-1}) \times
\st_{t_{k+1}}(\chi_{v,k+1}) \times \cdots \times \st_{t_u}(\chi_{v,u})$ et
 
\item $\Xi:\frac{1}{2} \Zm \longrightarrow \overline \Zm_l^\times$ 
 est dÈfini par $\Xi(\frac{1}{2})=q_v^{\frac{1}{2}}$. 
\end{itemize}
\end{itemize}
\end{prop}

\begin{proof}
Il s'agit ‡ nouveau de la proposition 3.6 de \cite{boyer-aif} avec $s=1$ et $\pi_v$ la reprÈsentation triviale de
$F_v^\times$: avec les notation de loc. cit., $R_{\pi_{k,v}}(1,t_k)(h,0)$ disparait i.e. c'est la reprÈsentation
triviale de $GL_0(F_v)$.
\end{proof}

\rem ‡ partir de ces descriptions cohomologiques en niveau infini, on retrouve leurs versions en niveau
fini, et maximal en $v$, en utilisant la proposition \ref{prop-preHecke}. En particulier dans la formule de
la proposition, on prend les invariants sous $I$
de $\Pi^{\oo,v} \otimes \Bigl ( \Pi_v^{(k)} \times \speh_h(\chi_{v,k}) \Bigr )$.

\section{Localisation de la cohomologie}

\begin{nota} \label{nota-spl2}
Pour $I$ un niveau fini, soit
$$\Tm_I:=\overline \Zm_l \bigl [T_{w,i}:~w \in \Spl(I) \hbox{ et } i=1,\cdots,d \bigr ],$$
l'algËbre de Hecke associÈe ‡ $\Spl(I)$, o˘ 
 $T_{w,i}$ est la fonction caractÈristique de
$$GL_d(\OC_w) \diag(\overbrace{\varpi_w,\cdots,\varpi_w}^{i}, \overbrace{1,\cdots,1}^{d-i} ) 
GL_d(\OC_w) \subset  GL_d(F_w).$$
\end{nota}

Le rÈsultat suivant tirÈ de \cite{emerton-gee} est la relation d'Eichler-Shimura dÈmontrÈe par
Wedhorn dans \cite{wed}, dans le cadre de nos variÈtÈs de Shimura de Kottwitz-Harris-Taylor.

\begin{theo} (cf. \cite{emerton-gee} 3.3.1) \label{theo-congruence}
Pour tout $w \in \Spl(I)$, l'action de
$$\sum_{i=0}^d(-1)^i q_w^{\frac{i(i-1)}{2}} T_{w,i} \frob_w^{d-i}$$
sur chacun des $H^j(X_{I,\bar \eta},V_{\xi,\overline \Zm_l})$ est nulle.
\end{theo}

Dans la suite, on fixe une place $v \in \Spl$, un idÈal $I \in \IC_v$ tel que $I=I^v$ 
\footnote{En particulier, on a $v \in \Spl(I)$.}
ainsi qu'un idÈal maximal $\mathfrak{m}$ de $\Tm_I$ de corps rÈsiduel 
$\overline \Fm_l$ tel qu'il existe  un entier $1 \leq h \leq d$ et  $i \in \Zm$ tels que
\addtocounter{smfthm}{1}
\begin{equation} \label{eq-m}
H^i(X^{\geq h}_{I,\bar s_v},V_{\xi,\overline \Zm_l})_{\mathfrak m} \neq (0).
\end{equation}
Pour tout $w \in \Spl(I)$, on note
$$P_{\mathfrak{m},w}(X):=\sum_{i=0}^d(-1)^i q_w^{\frac{i(i-1)}{2}} \overline{T_{w,i}} X^{d-i} \in \overline 
\Fm_l[X]$$
le polynÙme de Hecke associÈ ‡ $\mathfrak m$ et 
$$
S_{\mathfrak{m}}(w) := \bigl \{ \lambda \in \Tm_I/\mathfrak m \simeq \overline \Fm_l \hbox{ tel que }
P_{\mathfrak{m},w}(\lambda)=0 \bigr \} ,$$
le multi-ensemble des paramËtres de Satake modulo $l$ en $w$ associÈs ‡ $\mathfrak m$.

\rem on rappelle qu'un multi-ensemble est un couple $(A,m)$ o˘ $A$ est un ensemble appelÈ le support et
$m:A \longrightarrow \Nm^* \cup \{ + \oo \}$ est la multiplicitÈ au sens o˘ $a \in A$ apparait $m(a)$ fois 
dans le multi-ensemble $(A,m)$. On dira qu'un multi-ensemble $(A,m)$ est contenu dans $(A',m')$ si et 
seulement si $A \subset A'$ et pour tout $a \in A$, on a $m(a) \leq m'(a)$.

Avec les notations prÈcÈdentes, l'image $\overline{T_{w,i}}$ de $T_{w,i}$ dans  $\Tm_I/\mathfrak m$
s'Ècrit 
$$\overline{T_{w,i}}=q_w^{\frac{i(1-i)}{2}} \sigma_i (\lambda_1,\cdots, \lambda_d)$$
o˘ $S_{\mathfrak m}(w)=\{ \lambda_1,\cdots,\lambda_d \}$ et $\sigma_i$ dÈsigne la $i$-Ëme fonction
symÈtrique ÈlÈmentaire. 

\begin{nota} \label{nota-mdual}
On notera alors $\mathfrak m^\vee$ l'idÈal maximal de $\Tm_I$ dÈfini par
$$T_{w,i} \in \Tm_I \mapsto q_w^{\frac{i(1-i)}{2}} \sigma_i (\lambda_1^{-1},\cdots,\lambda_d^{-1})
\in \overline \Fm_l.$$
\end{nota}

\begin{defi}
On dÈfinit
$$l_{\mathfrak{m}}(w;\alpha):= \max \Bigl \{ s \hbox{ tel que } \{ \alpha,q_w \alpha,\cdots,q_w ^{s-1} \alpha \}
\subset S_{\mathfrak{m}}(w) \Bigr \}$$
et
$$l_{\mathfrak{m}}(w):= \max_{\alpha \in S_{\mathfrak{m}}(w)} l_{\mathfrak{m}}(w;\alpha).$$
\end{defi}

\rem dans la dÈfinition prÈcÈdente, $\{ \alpha,q_w \alpha,\cdots,q_w^{s-1} \alpha \}$ est considÈrÈ comme un
multi-ensemble et l'inclusion associÈe est relative aux multi-ensembles. En particulier si 
$q_w \equiv 1 \mod l$,
alors $l_{\mathfrak{m}}(w;\alpha)$ est simplement la multiplicitÈ de $\alpha$ dans $S_{\mathfrak{m}}(w)$.

\begin{lemm} \label{lem-preHecke2}
Avec les notations prÈcÈdentes, s'il existe $i$ avec
$$H^i(X_{I,\bar s_v}^{\geq h},V_{\xi,\overline \Zm_l})_{\mathfrak{m}} \otimes_{\overline \Zm_l} 
\overline \Qm_l \neq (0) \quad \hbox{resp. } 
H^i(X_{I,\bar s_v}^{\geq h},j^{\geq h}_! V_{\xi,\overline \Zm_l})_{\mathfrak{m}}
\otimes_{\overline \Zm_l} \overline \Qm_l \neq (0),$$
alors $l_{\mathfrak m}(v) \geq h$.
\end{lemm}

\begin{proof}
Rappelons que l'action du facteur $GL_h(F_v)$ du LÈvi
$P_{h,d}(F_v)$ sur ${\displaystyle \lim_{\atop{\rightarrow}{I \in \IC_v}}} 
H^i(X_{I,\bar s_v,1}^{\geq h}, V_{\xi,\bar \Qm_l})$
(resp. ${\displaystyle \lim_{\atop{\rightarrow}{I \in \IC_v}}} H^i(X_{I,\bar s_v,1}^{=h},j^{\geq h}_{1,!} V_{\xi,\bar \Qm_l})$)
se factorise par $\val \circ \det: GL_{h}(F_v) \longrightarrow \Zm$ de sorte que le rÈsultat dÈcoule via
(\ref{eq-Ih1}) (resp. de (\ref{eq-Ih2})) de la proposition \ref{prop-preHecke}, 
en remarquant que les paramËtres de Satake des invariants
sous $GL_h(F_v)$ de la reprÈsentation triviale $\speh_h(1_v)$ sont $\{ q_v^{\frac{1-h}{2}}, q_v^{\frac{3-h}{2}},
\cdots , q_v^{\frac{h-1}{2}} \}$.

\end{proof}

\begin{theo} \label{theo1}
Si $l_{\mathfrak{m}}(v)=1$ alors pour toute reprÈsentation algÈbrique $\xi$, 
la localisation $H^i(X_{I,\bar \eta},V_{\xi,\overline \Zm_l})_{\mathfrak{m}}$ en $\mathfrak{m}$
de la cohomologie de $V_{\xi,\overline \Zm_l}$ est nulle pour $i \neq d-1$ et sans torsion, pour $i=d-1$.
\end{theo}

\rem 
dans l'ÈnoncÈ prÈcÈdent, il faut voir la place $v$ comme une place auxiliaire au sens o˘, pour $I$ fixÈ,
dËs qu'il existe $v \in \Spl(I)$ avec $I=I^v \in \IC_v$ telle que $l_{\mathfrak{m}}(v)=1$, alors la localisation
de la cohomologie est sans torsion concentrÈe en degrÈ mÈdian.

\begin{proof}
Nous allons montrer par rÈcurrence sur $h$ de $d$ ‡ $2$ que les 
$H^i(X_{I,\bar s_v}^{\geq h},V_{\xi,\overline \Zm_l})_{\mathfrak{m}}$ sont nuls: d'aprËs le lemme prÈcÈdent
c'est dÈj‡ vrai pour les parties libres, il ne reste donc plus qu'‡ considÈrer la torsion de ces groupes.
Pour $h=d$, les $H^i(X_{I,\bar s_v}^{\geq d},V_{\xi,\overline \Zm_l})$ sont nuls pour $i \neq 0$ et sans 
torsion pour $i=0$, le rÈsultat dÈcoule donc du lemme prÈcÈdent.
Supposons le rÈsultat acquis jusqu'au rang $h+1$ et traitons le cas de $h \geq 2$. ConsidÈrons 
la suite exacte courte de faisceaux pervers sans torsion\footnote{Les strates 
$X^{\geq h}_{I,\bar s_v}$ 
Ètant lisses et $j^{\geq h}$ Ètant affine, les trois termes de la suite exacte sont pervers et sont libres
au sens de la thÈorie de torsion naturelle issue de la structure $\overline \Zm_l$-linÈaire, cf. 
\cite{boyer-torsion} \S 1.1-1.3.}
$$
0 \rightarrow i_{h+1,*} V_{\xi,\overline \Zm_l,|X_{I,\bar s_v}^{\geq h+1}}[d-h-1] \longrightarrow 
j^{\geq h}_! j^{\geq h,*} V_{\xi,\overline \Zm_l,|X^{\geq h}_{I,\bar s_v}}[d-h] \longrightarrow 
V_{\xi,\overline \Zm_l,|X^{\geq h}_{I,\bar s_v}}[d-h] \rightarrow 0.
$$
Il rÈsulte du thÈorËme d'Artin, cf. par exemple le thÈorËme 4.1.1 de \cite{BBD} et, donc, de 
l'affinitÈ des strates $X^{=h}_{I,\bar s_v}$ d'aprËs le thÈorËme \ref{theo-ito}, que 
les $H^i(X_{I,\bar s_v}^{\geq h},j^{\geq h}_! j^{\geq h,*} 
V_{\xi,\overline \Zm_l,|X^{\geq h}_{I,\bar s_v}}[d-h])$ sont nuls pour $i<0$ et sans torsion pour $i=0$,
de sorte que pour $i>0$, on a
\addtocounter{smfthm}{1}
\begin{equation} \label{eq-sec}
0 \rightarrow H^{-i-1}(X^{\geq h}_{I,\bar s_v},V_{\xi,\overline \Zm_l}[d-h]) \longrightarrow
H^{-i}(X^{\geq h+1}_{I,\bar s_v},V_{\xi,\overline \Zm_l}[d-h-1]) \rightarrow 0,
\end{equation}
et pour $i=0$,
\addtocounter{smfthm}{1}
\begin{multline} \label{eq-sec0}
0 \rightarrow H^{-1}(X^{\geq h}_{I,\bar s_v},V_{\xi,\overline \Zm_l}[d-h]) \longrightarrow
H^{0}(X^{\geq h+1}_{I,\bar s_v},V_{\xi,\overline \Zm_l}[d-h-1]) \longrightarrow \\
H^0(X^{\geq h}_{I,\bar s_v},j^{\geq h}_! j^{\geq h,*} V_{\xi,\overline \Zm_l}[d-h] ) \longrightarrow
H^{0}(X^{\geq h}_{I,\bar s_v},V_{\xi,\overline \Zm_l}[d-h]) \rightarrow \cdots
\end{multline}

Ainsi aprËs localisation en $\mathfrak{m}$ et en utilisant l'hypothËse de rÈcurrence, on obtient que les 
$H^i(X_{I,\bar s_v}^{\geq h}, V_{\xi,\overline \Zm_l}[d-h])_{\mathfrak{m}}$ sont nuls pour $i<0$
et sans torsion pour $i=0$. En utilisant que la propriÈtÈ $l_{\mathfrak{m}}(v)=1$ est invariante par dualitÈ,
i.e.
$$l_{\mathfrak{m}}(v)=1 \Leftrightarrow l_{\mathfrak{m}^\vee}(v)=1,$$
on en dÈduit alors par application de la dualitÈ de Verdier que les 
$H^i(X_{I,\bar s_v}^{\geq h}, V_{\xi,\overline \Zm_l}[d-h])_{\mathfrak{m}}$ sont nuls pour $i \neq 0$
et sans torsion pour $i=0$. On est ainsi ramener sur $\overline \Qm_l$ en degrÈ mÈdian o˘ le rÈsultat
dÈcoule du lemme prÈcÈdent.
 
Les mÍmes arguments appliquÈs au cas $h=1$, nous donnent que les 
$H^i(X_{I,\bar s_v},V_{\xi,\overline \Zm_l})_{\mathfrak{m}}$ sont nuls pour $i \neq d-1$
et sans torsion pour $i=d-1$. Le thÈorËme de changement de base lisse fournit 
$H^i(X_{I,\bar \eta},V_{\xi,\overline \Zm_l}) \simeq H^i(X_{I,\bar s_v}^{\geq 1},V_{\xi,\overline \Zm_l})$,
d'o˘ le rÈsultat.

\end{proof}

Une analyse plus fine de la preuve prÈcÈdente permet d'obtenir la prÈcision suivante.

\begin{prop} \label{prop1}
Soit $i \geq 0$ tel que la torsion de $H^{-i}(X_{I,\bar \eta},V_{\xi,\overline \Zm_l}[d-1])_{\mathfrak{m}}$
est non nulle. On a alors $l_{\mathfrak{m}}(v) \geq i+2$.
\end{prop}

\rem l'inÈgalitÈ Èvidente $l_{\mathfrak{m}}(v) \leq d$, nous donne en particulier que la torsion de
$H^{0}(X_{I,\bar \eta},V_{\xi,\overline \Zm_l})_{\mathfrak{m}}$ est nulle; on peut donc
comprendre l'ÈnoncÈ prÈcÈdent comme une gÈnÈralisation de ce fait ÈlÈmentaire.

\begin{proof}
Notons $r=l_{\mathfrak{m}}(v)$. En reprenant la preuve du thÈorËme prÈcÈdent,
on obtient que les $H^i(X^{\geq h}_{I,\bar s_v},V_{\xi,\overline \Zm_l}[d-h])_{\mathfrak{m}}$ sont nuls
pour tout $i$ (resp. $i \neq 0$ et sans torsion pour $i=0$) tant que $h>r$ (resp. pour $h=r$).
Montrons ‡ prÈsent par rÈcurrence sur $h$ de $r$ ‡ $1$ que les 
$H^{-i}(X^{\geq h}_{I,\bar s_v},V_{\xi,\overline \Zm_l}[d-h])_{\mathfrak{m}}$ sont nuls pour $i>r-h$ et 
sans torsion pour $i=r-h$. D'aprËs les isomorphismes de (\ref{eq-sec}), on voit que
la nullitÈ de $H^{-i}(X^{\geq h+1}_{I,\bar s_v},V_{\xi,\overline \Zm_l}[d-h-1])_{\mathfrak m}$
pour tout $i>r-h-1$ donne celle de $H^{-i}(X^{\geq h}_{I,\bar s_v},V_{\xi,\overline \Zm_l}[d-h])_{\mathfrak m}$
pour tout $i>r-h$. Le rÈsultat de l'ÈnoncÈ dÈcoule, par contraposition, 
du cas $h=1$ et du changement de base lisse.

\end{proof}

Dans l'argument prÈcÈdent on voit que la torsion peut apparaÓtre, par exemple, ‡ cause de la flËche
$$H^0(X^{\geq r}_{I,\bar s_v},V_{\xi,\overline \Zm_l,|X_{I,\bar s_v}^{\geq r}}[d-r])
\hookrightarrow H^0(X^{\geq r-1}_{I,\bar s_v},j^{\geq r-1}_! j^{\geq r-1,*} 
V_{\xi,\overline \Zm_l,|X^{\geq r-1}_{I,\bar s_v}}[d-r+1])$$
entre deux $\overline \Zm_l$-modules libres ‡ priori non nuls. 
Les exemples de l'introduction o˘ la torsion
est non nulle, illustrent le fait que pour certains $\mathfrak m$, l'injection prÈcÈdente est non stricte.

\rem Sur $\overline \Qm_l$, 
d'aprËs la proposition \ref{prop-non-temperee}, la 
$\{ \Pi^{\oo,v} \}$-composante d'une telle flËche est nulle si $\Pi$ n'est pas tempÈrÈe.

\begin{defi} \label{defi-ref0}
Pour $1 \leq \delta \leq l_{\mathfrak{m}}(v)$, on dÈfinit
$$\mu_{\mathfrak{m}}(v;\delta)= \sharp \Bigl \{ \alpha \in S_{\mathfrak{m}}(v):~ 
l_{\mathfrak{m}}(v;\alpha) \geq \delta. 
\Bigr \}.$$
\end{defi}

Supposons que $\mathfrak{m}$ et $\xi$ sont tels qu'il existe $i \leq 0$ tel que la torsion de
$H^i(X_{I,\bar \eta},V_{\xi,\overline \Zm_l}[d-1])_{\mathfrak{m}}$ est non nulle.

\begin{nota}
Soit $i_{\mathfrak{m},\xi} \geq 0$ maximal tel que pour tout $i<-i_{\mathfrak{m},\xi}$, la torsion de
$H^i(X_{I,\bar \eta},V_{\xi,\overline \Zm_l}[d-1])_{\mathfrak{m}}$ est nulle.
\end{nota}

\begin{lemm}
Pour tout $1 \leq h \leq 1+i_{\mathfrak{m},\xi}$, le localisÈ $H^i(X_{I,\bar \eta}^{\geq h},
V_{\xi,\overline \Zm_l,|X^{\geq h}_{I,\bar s_v}}[d-h])_{\mathfrak{m}}$ est sans torsion pour
$i<h-1-i_{\mathfrak{m},\xi}$ et de torsion non nulle pour $i=h-1-i_{\mathfrak{m},\xi}$.
\end{lemm}

\begin{proof}
On raisonne par rÈcurrence sur $h$ de $1$ ‡ $1+i_{\mathfrak{m},\xi}$: le cas $h=1$ dÈcoule de la
dÈfinition de $i_{\mathfrak{m},\xi}$ et du changement de base lisse. La propriÈtÈ d'inductivitÈ de $h-1$ ‡ $h$
se dÈduit alors des isomorphismes (\ref{eq-sec}) localisÈs en $\mathfrak m$ et, pour 
$h=1+i_{\mathfrak m,\xi}$, de la suite exacte longue (\ref{eq-sec0}).

\end{proof}

\begin{prop} \label{prop2}
L'action du Frobenius $\frob_v$ sur le $\overline \Fm_l$-espace 
$$
H_{tor}^{-i_{\mathfrak{m},\xi}}(X_{I,\bar \eta},V_{\xi,\overline \Zm_l}[d-1])_{\mathfrak{m}} 
\otimes_{\overline \Zm_l} \overline \Fm_l
$$ 
admet au plus $\mu_{\mathfrak{m}}(v;i_{\mathfrak{m},\xi}+2)$ valeurs propres distinctes.
\end{prop}

\rem comme prÈcÈdemment la proposition ci-dessus est valable pour toute place $v$ telle que 
$I=I^v \in \IC_v$.

\begin{proof}
Continuons les arguments de la preuve du lemme prÈcÈdent pour les $h>1+i_{\mathfrak m,\xi}$. 
Des isomorphismes (\ref{eq-sec}) et de la suite exacte (\ref{eq-sec0}),
on en dÈduit, par rÈcurrence sur $h>1+i_{\mathfrak m,\xi}$, que les $H^i(X_{I,\bar s_v}^{\geq h},
V_{\xi,\overline \Zm_l,|X^{\geq h}_{I,\bar s_v}}[d-h])_{\mathfrak{m}}$ sont sans torsion pour tout $i \leq 0$.
En outre pour $h=1+i_{\mathfrak m,\xi}$, la torsion de
$H^0(X_{I,\bar s_v}^{\geq 1+i_{\mathfrak{m},\xi}},V_{\xi,\overline \Zm_l}
[d-1-i_{\mathfrak{m}, xi}])_{\mathfrak{m}}$ s'obtient comme celle du conoyau  
\begin{multline*}
H^0(X_{I,\bar s_v}^{\geq 2+i_{\mathfrak{m},\xi}},V_{\xi,\overline \Zm_l} 
[d-2-i_{\mathfrak{m},\xi}])_{\mathfrak{m}}  \longrightarrow \\
H^0(X_{I,\bar s_v}^{\geq 1+i_{\mathfrak{m},\xi}},
j^{\geq 1+i_{\mathfrak{m},\xi}}_! j^{\geq 1+i_{\mathfrak{m},\xi},*} V_{\xi,\overline \Zm_l}
[d-1-i_{\mathfrak{m},\xi}])_{\mathfrak{m}}.
\end{multline*}
En utilisant les isomorphismes (\ref{eq-sec}) respectivement pour $h=i_{\mathfrak m,\xi},\cdots,1$
avec $i=i_{\mathfrak m,\xi}-h$, on obtient que la torsion de 
$H^{-i_{\mathfrak{m},\xi}}(X_{I,\bar \eta},V_{\xi,\overline \Zm_l}[d-1])_{\mathfrak{m}}$
s'obtient comme celle du conoyau du morphisme prÈcÈdent.

Ainsi d'aprËs la proposition \ref{prop-non-temperee}, et en utilisant la remarque (ii) qui suit
\ref{nota-hixi}, les valeurs propres de $\frob_v$ cherchÈes sont ‡ 
prendre dans la rÈduction modulo $l$ des $[H^0(2+i_{\mathfrak m,\xi},\xi) ] \{ \Pi^{\oo,v} \}$, pour 
$\Pi$ tempÈrÈe d'aprËs la remarque prÈcÈdant \ref{defi-ref0}, 
telles que pour tout $v \in \Spl(I)$, les paramËtres de Satake en $v$ modulo $l$
sont donnÈs par $\mathfrak m$. Pour une telle reprÈsentation $\Pi$ avec, cf. la proposition
\ref{prop-temperee-explicite}, 
$$\Pi_v \simeq \st_{t_1}(\chi_{v,1}) \times \cdots \times \st_{t_u}(\chi_{v,u}),$$
avec les notations de la proposition \ref{prop-temperee-explicite}, pour avoir des 
vecteurs invariants sous $GL_{d-2-i_{\mathfrak m,\xi}}(\OC_v)$, il faut 
\begin{itemize}
\item qu'il existe $1 \leq k \leq u$ tel que $t_k=2+i_{\mathfrak m,\xi}$,

\item que pour tout $1 \leq i \neq k \leq u$, on ait $t_i=1$ et

\item que les caractËres $\chi_{v,1}, \cdots, \chi_{v,u}$ soient non ramifiÈs.
\end{itemize}
Pour un tel $\Pi$, d'aprËs la proposition \ref{prop-temperee-explicite}, la valeur propre de $\frob_v$ 
associÈe est $\chi_{v,k}(\varpi_v)q_v^{\frac{d-i_{\mathfrak m,\xi}-2}{2}}$, et d'aprËs la proposition
\ref{prop-preHecke}, le multi-ensemble des paramËtres de Satake est
\begin{multline*}
\Bigl \{ \chi_{v,1}(\varpi_v), \cdots, \chi_{v,k-1}(\varpi_v), \\ \chi_{v,k}(\varpi_v) 
q_v^{-\frac{i_{\mathfrak m,\xi}+1}{2}}, 
\chi_{v,k}(\varpi_v) q_v^{-\frac{i_{\mathfrak m,\xi}-1}{2}}, \cdots, \chi_{v,k}(\varpi_v) 
q_v^{\frac{1+i_{\mathfrak m,\xi}}{2}}, \\
\chi_{v,k+1}(\varpi_v),\cdots,\chi_{v,u}(\varpi_v) \Bigr \}.
\end{multline*}
En particulier modulo $l$, ce multi-ensemble de paramËtres de Satake contient un sous-multi-ensemble
de la forme $\{ \alpha, q_v \alpha, \cdots, q_v^{1+i_{\mathfrak m,\xi}} \alpha \}$. On obtient ainsi,
par dÈfinition, au plus
$\mu_{\mathfrak{m}}(v;i_{\mathfrak{m},\xi}+2)$ valeurs propres de Frobenius distinctes.

\end{proof}

\begin{coro} \label{coro-dimension}
On suppose $l \geq d+2$ et supposons que l'ordre de $q_v$ modulo $l$ soit $>d$.
Si $\rho$ est une $\overline \Fm_l$-sous-reprÈsentation galoisienne irrÈductible
dans la torsion de $H^{-i_{\mathfrak{m},\xi}}(X_{I,\bar \eta},V_{\xi,\overline \Zm_l}[d-1])_{\mathfrak{m}}$ 
telle que les valeurs propres de $\frob_v$ sont distinctes alors
$$\dim \rho \leq d-1-i_{\mathfrak{m},\xi}.$$
\end{coro}

\rem en particulier on retrouve un phÈnomËne bien connu en caractÈristique nulle, ‡ savoir que la dimension chute ‡ mesure
qu'on s'Èloigne du degrÈ mÈdian.

\begin{proof}
Notons $V$ (resp. $S$) l'ensemble des valeurs propres de 
$q_v^{-\frac{d-i_{\mathfrak m,\xi}-2}{2}} \rho(\frob_w)$ (resp. les paramËtres de Satake modulo $l$ associÈs 
‡ $\mathfrak m$). D'aprËs la proposition prÈcÈdente si $\lambda \in V$ alors 
$\{ \lambda, q_v \lambda, \cdots, q_v^{i_{\mathfrak m,\xi}+1} \lambda \} \subset S$. En particulier, les valeurs propres
Ètant supposÈes distinctes, on a $V \subset S$ et donc le cardinal de $V$ est $\leq d$.
Notons alors que, comme l'ordre de $q_v$ est $>d$, alors $q_vV$ n'est pas inclus dans $V$: en effet sinon 
on aurait $q_v^n V \subset V$ et $V$ contiendrait un ensemble de la forme 
$\{ \alpha,q_v \alpha,\cdots, q_v^d \alpha \}$ ce qui n'est pas car $V$ est de cardinal $\leq d$.
Prenons alors $\lambda_0 \in V$ tel que $\lambda_0 q_v \not \in V$,
i.e. $\lambda_0 q_v^{2+i_{\mathfrak m,\xi}} \not \in S$.
On en dÈduit alors que pour tout $k=1,\cdots,1+i_{\mathfrak m,\xi}$, l'ÈlÈment $\lambda_0 q_v^{k}$
appartient ‡ $S$ mais pas ‡ $V$, d'o˘ le rÈsultat.
\end{proof}

%

Le rÈsultat suivant est l'analogue, dans le cadre restrictif des variÈtÈs de Shimura
simples de Kottwitz-Harris-Taylor, du thÈorËme principal de \cite{scholze-torsion}, cf. aussi le 
thÈorËme 6.3.1 de \cite{scholze-cara}.

\begin{theo} \label{theo-3}
Soit $\mathfrak m$ vÈrifiant (\ref{eq-m}), il existe alors une reprÈsentation continue
$$\overline \rho_{\mathfrak m}: G_F \longrightarrow GL_d(\overline \Fm_l)$$
non ramifiÈe ‡ toutes les places ne divisant pas $I$ et telle que pour tout $w \in \Spl(I)$, le polynÙme
caractÈristique de $\overline \rho_{\mathfrak m}(\frob_w)$ est
$$P_{\mathfrak m,w}(X):=
\sum_{i=0}^d(-1)^i q_w^{\frac{i(i-1)}{2}} \overline{T_{w,i}} X^{d-i} \in \Tm_I/\mathfrak m \simeq 
\overline \Fm_l.$$
\end{theo}

\begin{proof}
On choisit une place $v$ telle que, avec les notations prÈcÈdentes, $I=I^v \in \IC_v$, et on reprend
la preuve du thÈorËme \ref{theo1}. Afin de formaliser l'argument introduisons
la notion suivante: on dira d'un $\Tm_{I}$-module $M$ qu'il vÈrifie la propriÈtÈ \textbf{(P)},
s'il admet une filtration finie 
$$(0)=\Fil^0(M) \subset \Fil^1(M) \cdots \subset \Fil^r(M)=M$$
telle que pour tout $k=1,\cdots,r$, il existe
\begin{itemize}
\item une reprÈsentation automorphe $\Pi_k$ irrÈductible et entiËre
de $G(\Am)$, apparaissant dans la cohomologie de $X_{\IC,\overline \eta_v}$ ‡ coefficients dans 
$V_{\xi,\overline \Qm_l}$,

\item une reprÈsentation irrÈductible entiËre $\widetilde \Pi_{k,v}$ de mÍme support cuspidal
que $\Pi_{k,v}$

\item et un $\Tm_{I}$-rÈseau stable $\Gamma$ de 
$(\Pi^{\oo,v}_k)^{I^v} \otimes \widetilde \Pi_{k,v}^{GL_d(\OC_v)}$ tel que
\begin{itemize}
\item soit $\gr^k(M)$ est libre et isomorphe ‡ $\Gamma$,

\item soit $\gr^k(M)$ est de torsion et un sous-quotient de $\Gamma/\Gamma'$ pour 
$\Gamma' \subset \Gamma$ un deuxiËme $\Tm_I$-rÈseau stable.
\end{itemize}
\end{itemize}
La propriÈtÈ  \textbf{(P)} est clairement stable par extensions, par sous-quotients et, en remplaÁant la 
condition $\xi$-cohomologique par $\xi^\vee$-cohomologique, par dualitÈ.

Pour un tel $\Tm_I$-module $M$ et pour tout $k$ tel que $\gr^k(M)$ non nul, le systËme
de paramËtres de Satake modulo $l$ de $\gr^k(M) \otimes_{\overline \Zm_l} \overline \Fm_l$
est la rÈduction modulo $l$ d'un systËme sur $\overline \Zm_l$ associÈ ‡ une reprÈsentation automorphe
apparaissant dans la cohomologie de $X_{\IC,\overline \eta_v}$ ‡ coefficients dans 
$V_{\xi,\overline \Qm_l}$. D'aprËs \cite{h-t} ‡ cette reprÈsentation automorphe est associÈe une reprÈsentation galoisienne
dont la rÈduction modulo $l$ sera telle qu'aux places non ramifiÈes, les frobenius auront pour valeurs propres
les paramËtres de Satake modulo $l$ donnÈs en une telle place par $\gr^k(M) \otimes_{\overline \Zm_l} \overline \Fm_l$.
Ainsi il suffit de montrer que les $H^i(X^{\geq 1}_{I,\bar s_v},V_{\xi,\overline \Zm_l})$ vÈrifient la propriÈtÈ \textbf{(P)}.

Pour ce faire nous allons montrer par rÈcurrence sur $h$ de $d$ ‡ $1$, que les $H^i(X^{\geq h}_{I,\bar s_v},V_{\xi,\overline \Zm_l})$ 
vÈrifient la propriÈtÈ \textbf{(P)}.
On sait dÈj‡ que c'est le cas pour les parties libres, cf. la proposition \ref{prop-temperee-explicite}, 
et donc c'est vrai pour $h=d$.
Supposons donc le rÈsultat acquis jusqu'au rang $h+1$ et considÈrons la suite exacte courte
$$
0 \rightarrow i_{h+1,*} V_{\xi,\overline \Zm_l,|X_{I,\bar s_v}^{\geq h+1}}[d-h-1] \longrightarrow 
j^{\geq h}_! j^{\geq h,*} V_{\xi,\overline \Zm_l,|X^{\geq h}_{I,\bar s_v}}[d-h] \longrightarrow 
V_{\xi,\overline \Zm_l,|X^{\geq h}_{I,\bar s_v}}[d-h] \rightarrow 0.
$$
D'aprËs l'hypothËse de rÈcurrence, les groupes de cohomologie de 
$i_{h+1,*} V_{\xi,\overline \Zm_l,|X_{I,\bar s_v}^{\geq h+1}}[d-h-1]$ vÈrifient \textbf{(P)} ainsi que
ceux de $j^{\geq h}_! j^{\geq h,*} V_{\xi,\overline \Zm_l,|X^{\geq h}_{I,\bar s_v}}[d-h]$ en degrÈ $i \leq 0$
puisqu'ils sont soit nuls pour $i<0$, soit sans torsion pour $i=0$. On en dÈduit alors
que les groupes de cohomologie de $V_{\xi,\overline \Zm_l,|X^{\geq h}_{I,\bar s_v}}[d-h]$
vÈrifient \textbf{(P)} en degrÈ $i \leq 0$ et donc, par dualitÈ, pour tout $i$.

\end{proof}

\rem moralement notre preuve est dans l'esprit trËs proche de celle du thÈorËme 6.3.1 de \cite{scholze-cara}, i.e.
toute classe de torsion dans la cohomologie de $X_{I}$ se relËve en caractÈristique nulle dans la cohomologie
en degrÈ mÈdian d'une variÈtÈ d'Igusa.\footnote{Pour une telle propriÈtÈ le choix de la place $v \in \Spl(I)$ est libre et il n'est pas
difficile de gÈnÈraliser ce fait pour toute place $v \in \Spl$.}
En particulier on n'obtient pas de \og nouveaux \fg{} systËmes de paramËtres de Satake, ce qui d'aprËs \cite{scholze-cara}
semble Ítre un phÈnomËne partagÈ par les variÈtÈs de Shimura contrairement ‡ ce qui se passe dans le cas gÈnÈral, cf.
\cite{scholze-torsion}.

\begin{hypo} \label{hypo-theta}
Si $\theta: G_F \longrightarrow GL_n(\overline \Fm_l)$ est une
reprÈsentation irrÈductible continue telle que pour tout $v \in \Spl(I)$, on a 
$$P_{\mathfrak m,v}(\theta(\frob_v))=0 \qquad 
\hbox{resp. } P_{\mathfrak m^\vee,v}(\theta(\frob_v))=0$$
alors $\theta$ est Èquivalent ‡ $\overline \rho_{\mathfrak m}$ (resp. ‡ $\overline \rho_{\mathfrak m^\vee}$).
\end{hypo}

\rem d'aprËs \cite{emerton-gee}, l'hypothËse \ref{hypo-theta} est vÈrifiÈe si 
\begin{itemize}
\item soit $\overline \rho_{\mathfrak m}$ est induit d'un caractËre de $G_K$ pour $K/F$ une extension
galoisienne cyclique;

\item soit $l \geq d$ et $SL_d(k) \subset \overline \rho_{\mathfrak m}(G_F) \subset \overline \Fm_l^\times
GL_d(k)$ pour un sous-corps $k \subset \overline \Fm_l$.
\end{itemize}

\begin{theo} \label{theo2}
Supposons que $\overline \rho_{\mathfrak m}$ vÈrifie l'hypothËse \ref{hypo-theta} et que $l \geq d+2$. 
Alors les $H^i(X_{I,\bar \eta},V_{\xi,\overline \Zm_l})_{\mathfrak m}$ sont sans torsion.
\end{theo}

\rem on obtient ainsi une version amÈliorÈe du thÈorËme  3.4.2 de \cite{emerton-gee}. 

\begin{proof}
En utilisant le fait que $\mathfrak m^\vee$ vÈrifie, par hypothËse, 
la mÍme condition que $\mathfrak m$, il suffit de montrer 
que les $H^i(X_{I,\bar \eta}, V_{\xi,\overline \Zm_l}[d-1])_{\mathfrak m}$ ou de maniËre Èquivalente,
que les $H^i(X_{I,\bar \eta}, V_{\xi,\overline \Zm_l}[d-1])_{\mathfrak m^\vee}$,
sont sans torsion pour $i \leq 0$:
en effet s'il existait $i >0$ tel que la torsion de 
$H^i(X_{I,\bar \eta}, V_{\xi,\overline \Zm_l}[d-1])_{\mathfrak m}$ Ètait non nulle, alors par dualitÈ
celle de $H^{-i+1}(X_{I,\bar \eta}, V_{\xi,\overline \Zm_l}[d-1])_{\mathfrak m^\vee}$ serait aussi non nulle.

On raisonne alors par l'absurde et on note comme prÈcÈdemment $i_{\mathfrak m,\xi} \geq 0$ le plus
grand indice $i$ tel que la torsion de $H^{-i}(X_{I,\bar \eta},V_{\xi,\overline \Zm_l}[d-1])_{\mathfrak m}$
est non nulle. Soit alors $\theta$ une reprÈsentation de $G_F$ obtenue comme sous-quotient irrÈductible
de la torsion de $H^{-i_{\mathfrak m,\xi}} (X_{I,\bar \eta},V_{\xi,\overline \Zm_l}[d-1])_{\mathfrak m}$.
D'aprËs la relation de congruence \ref{theo-congruence} et l'hypothËse prÈcÈdente, on en dÈduit que
$\theta$ est isomorphe ‡ $\overline \rho_{\mathfrak m}$ et que donc, pour tout $v \in \Spl(I)$,
les valeurs propres de $\theta(\frob_v)$ correspondent aux paramËtres de Satake modulo $l$ donnÈs 
par $\mathfrak m$, ‡ multiplication par $q_v^{\frac{d-1}{2}}$ prËs.

Comme $l-1 > d$, on choisit $v \in \Spl(I)$ tel que l'ordre de $v_{|\Qm}$ modulo $l$ est 
$> d$. Soit alors $q_v^{\frac{d-1}{2}}\lambda$ une valeur propre de $\frob_v$ agissant sur la torsion de 
$H^{-i_{\mathfrak m,\xi}}(X_{I,\bar \eta},V_{\xi,\overline \Zm_l}[d-1])_{\mathfrak m}$: la rÈduction modulo
$l$ de $\lambda$ est un paramËtre de Satake modulo $l$ de $\mathfrak m$.
D'aprËs la preuve de la proposition \ref{prop2}, il existe alors 
une reprÈsentation automorphe $\Pi$ dont la composante $\Pi_v$ est de la forme
$$\Pi_v \simeq \st_{2+i_{\mathfrak m,\xi}}(\chi_{v,1}) \times \chi_{v,2} \times \cdots \times \chi_{v,u},$$
pour $\chi_{v,1}, \cdots, \chi_{v,u}$ des caractËres non ramifiÈs de $F_v^\times$, et
avec $\lambda=\chi_{v,1}(\varpi_v)q_v^{-\frac{i_{\mathfrak m,\xi}+1}{2}}$.
Les paramËtres de Satake modulo $l$ sont alors donnÈs par
$$
\Bigl \{ \chi_{v,1}(\varpi_v)q_v^{-\frac{i_{\mathfrak m,\xi}+1}{2}},
 \chi_{v,1}(\varpi_v)q_v^{-\frac{i_{\mathfrak m,\xi}-1}{2}}, \cdots,\chi_{v,1}
(\varpi_v)q_v^{\frac{i_{\mathfrak m,\xi}+1}{2}}, 
\chi_{v,2}(\varpi_v), \cdots \chi_{v,u}(\varpi_v) \Bigr \}.
$$
En particulier, on remarque que si $\lambda$ est un paramËtre de Satake modulo $l$ de $\mathfrak m$
‡ la place $v$ alors $q_v\lambda$ aussi. Comme par hypothËse l'ordre de $q_v$ modulo $l$ est $> d$,
alors $q_v^i \lambda$ serait un paramËtre de Satake modulo $l$ de $\mathfrak m$ ‡ la place $v$, 
pour tout $i=0,\cdots,d$, ce qui donnerait $d+1$ paramËtres distincts ce qui ne se peut pas, d'o˘ la 
contradiction.

\end{proof}

\bibliographystyle{siam}
\bibliography{bib-ok}

\end{document}